\documentclass[11pt]{amsart}

\usepackage{amssymb,latexsym}

\usepackage{graphicx,epsfig}
\usepackage[utf8]{inputenc}
\usepackage[dvipsnames]{xcolor}

\def\Bbb{\mathbb}

\def\S{\Bbb S}

\def\R{{\Bbb R}}
\def\CC{{\Bbb C}}	
\def\C{{\mathcal C}}	

\def\d{{\delta}}

\def \supp {\text{\rm car\,}}
\def \mass {\text{\rm mass\,}}

\def\e{\varepsilon}

\renewcommand{\epsilon}{\varepsilon}

\newcommand{\dd}{\mathrm{d}}

\def\Bbb{\mathbb}

\def\E{{\mathcal E}}

\newtheorem{thm}{Theorem}[section]

\newtheorem{cor}[thm]{Corollary}
\newtheorem{lemma}[thm]{Lemma}
\newtheorem{remark}[thm]{Remark}

\newtheorem{example}[thm]{Example}

\begin{document}

\title[Factorisation in Fourier Restriction Theory]{Factorisation in Fourier Restriction Theory and near Extremisers}
\date{\today}

\author[S. Buschenhenke]{Stefan Buschenhenke}
\address{Stefan Buschenhenke:  Mathematisches Seminar, C.A.-Universit\"at Kiel,
Heinrich-Hecht-Platz 6, D-24118 Kiel, Germany}
\email{{\tt buschenhenke@math.uni-kiel.de}}
\urladdr{http://www.math.uni-kiel.de/analysis/de/buschenhenke}

\begin{abstract} 
We give an alternative argument to the application of the so-called Maurey-Nikishin-Pisier factorisation in Fourier restriction theory. Based on an induction-on-scales argument, our comparably simple method applies to any compact quadratic surface, in particular compact parts of the paraboloid and the hyperbolic paraboloid. This is achieved by constructing near extremisers with big "mass", which itself might be of interest.
\end{abstract}

\maketitle



\section{Introduction}\label{intro}
\subsection{Historical background}
The restriction problem for the Fourier transform, introduced by Stein \cite{St1} in the seventies, asks for which values of $1\leq r,s\leq\infty$ the restriction of the Fouriertransform $Rf=\hat f|_S$ to a given submanifold $S\subset \R^{n+1}$ gives a bounded operator, i.e., for which values of $1\leq r,s\leq\infty$ is 
$$\left(\int_S |\hat f|^r\dd\sigma\right)^{1/r} \leq C \|f\|_{L^s(\R^{n+1})}$$
for all $f\in\mathcal{S}(\R^{n+1})$? Here, $\sigma$ denotes the surface measure of $S$, and for $r=\infty$, the left-hand-side has to be modified the obvious way.\\
In restriction theory, it is often useful to study the adjoint restriction operator associated to $S$, the \emph{extension operator} 
$$\E(f)(x,t)=\widehat{f\dd\sigma}(x,t).$$
For $S$ being a compact n-dimensional hypersurface with non-vanishing Gaussian curvature, the famous restriction conjecture proposed by Stein in the 70's is one of the hardest problems in harmonic analysis and has inspired many groundbreaking work.
In the language of the adjoint setting, the conjecture states that the extension operator is bounded from $L^q(S)$ to $L^p(\R^{n+1})$ if and only if $p>\frac{2n+2}{n}$ and $\frac{1}{q'}\geq\frac{n+2}{p}$ (Fig. 1).
Testing on the constant function, or respectively on a characteristic function of a small ball shows that both conditions are indeed necessary. 

\begin{figure}
\includegraphics[scale=0.6]{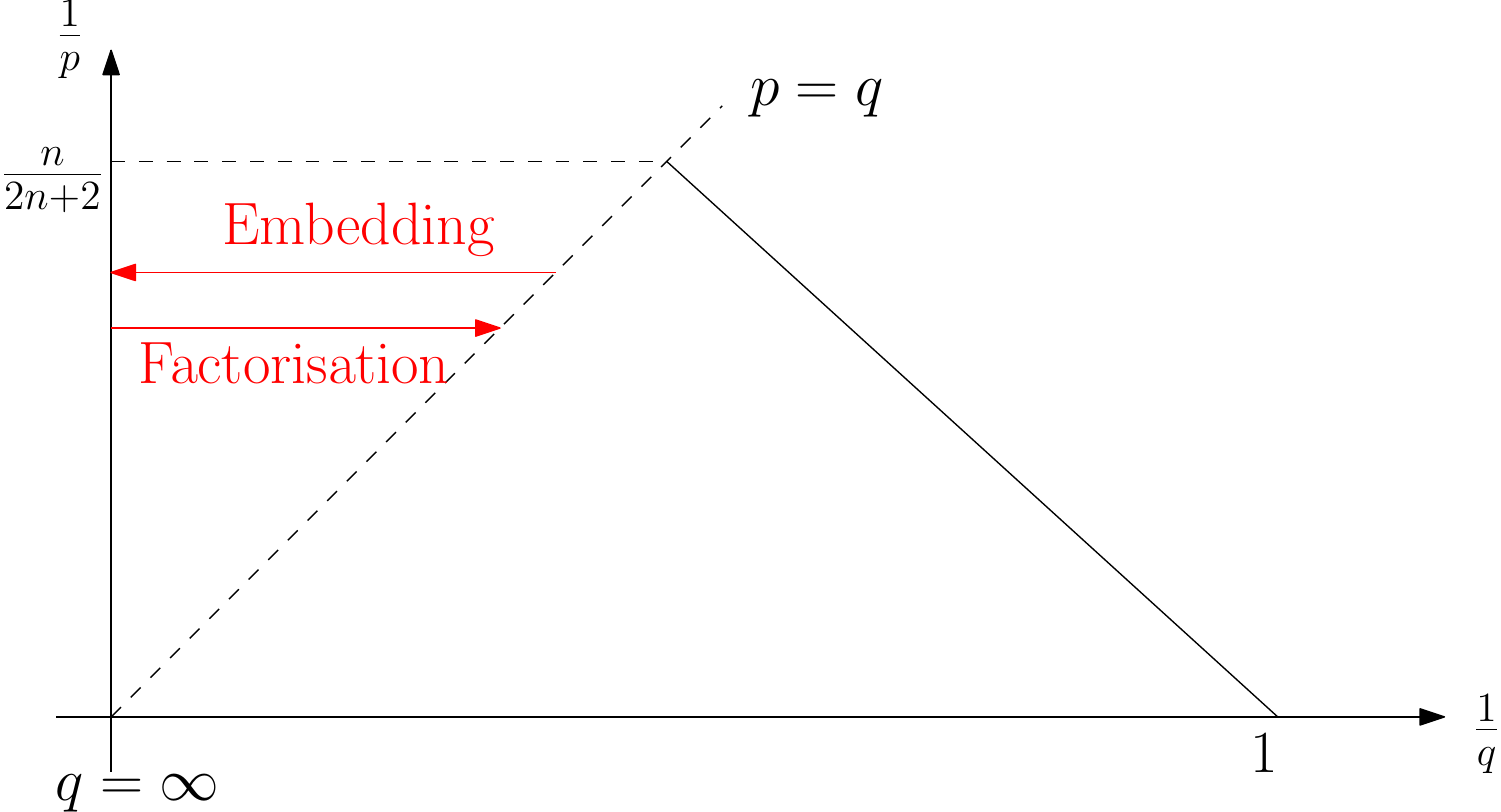}
\caption{Conjectured range for boundedness of the extension operator}
\end{figure}

For the important Hilbert space case $q=2$, the conjecture was confirmed  Stein and Tomas [To75], [St86]. Further important contributions were due to J. Bourgain in the nineties (\cite{Bo1}, \cite{Bo2}, \cite{Bo3}), and the introduction of the so-called bilinear method by Tao, Moyu, Vargas, Vega and Wulff (\cite{MVV1}, \cite{MVV2}, \cite{TVV} \cite{TV1}, \cite{TV2},  \cite{W2}, \cite{T2}). A new approach to the problem has been the study of multilinear variants of the extension operator in \cite{BCT}, which in turn imply estimates for the "standard" linear extension operator \cite{BoG}. Recently, a major step forward was achieved by the polynomial partitioning method, developed by L. Guth and N.H. Katz (\cite{GK}, \cite{Gu16},
\cite{Gu17}) (see also \cite{hr19} and \cite{WA18} for recent improvements).\\
Observe that the adjoint restriction estimate for $q=1$ and $p=\infty$ is trivial, and a $L^q(S)$-$L^p(\R^{n+1})$-estimate implies a $L^{\tilde q}(S)$-$L^p(\R^{n+1})$-estimate for any $\tilde q> q$ by Hölders inequality, since $S$ has finite measure.
On the other hand, the two critical lines $p=\frac{2n+2}{n}$ and $\frac{1}{q'}=\frac{n+2}{p}$ intersect on the diagonal at $q=p=\frac{2n+2}{n}$, and proving a estimate for $q=p$ arbitrary close to the critical value $\frac{2n+2}{n}$ is arguably the hardest case, as one then obtains $L^q(S)$-$L^p(\R^{n+1})$-boundedness for the open range $p>\frac{2n+2}{n}$ and $\frac{1}{q'}>\frac{n+2}{p}$ by interpolating with the trivial $L^1(S)$ to $L^\infty(\R^{n+1})$ estimate and by Hölder. It is even possible to obtain estimates for the endpoint line $\frac{1}{q'}=\frac{n+2}{p}$ by switching to a bilinear regime and interpolating with any bilinear estimate above the critical line, though we will not explain this in detail.\\
Now, to prove extension estimates on the diagonal, at least for $S$ being the sphere, it is indeed sufficient to prove $L^\infty(S)$-$L^p(\R^{n+1})$-estimates by using the Nikishin-Maurey-Pisier factorisation argument. 
To be more precise, for any fixed $p\geq 2$,
the following weak type estimate for the restriction operator $R$
\begin{align}\label{pinf}
	\| R f\|_{L^{1,\infty}(S)} \lesssim \|f\|_{L^{p'}(\R^{n+1})}
\end{align}
is equivalent to
\begin{align}\label{pp}
	\| R f\|_{L^{p',\infty}(S)} \lesssim \|f\|_{L^{p'}(\R^{n+1})}.
\end{align}
In other words, $R$ factors through $L^{p',\infty}(S)$, via composition with the embedding $L^{p',\infty}(S)\hookrightarrow L^{1,\infty}(S)$.
That \eqref{pp} implies \eqref{pinf} is trivial by embedding of Lorentz spaces, the converse is far from being obvious; this is what we call the factorisation argument (cf. Fig. 1) and for the sphere is due to Bourgain \cite{Bo1}.\\
Bascially, this factorisation argument consist of two ingredients, the actual, general Nikishin-Maurey-Pisier factorisation, plus a invariance consideration as final step.\\

The main part, the Nikishin-Maurey-Pisier factorisation argument, indeed works in a much more general setting, namely sublinear operators on Banach spaces of Rademacher type $p'$, but we will not dwell on that. For an excellent survey on this argument, we refer the interested reader to the book of Garcia-Cuerva and Rubio de Francia \cite{pgc}, which builds upon the work of Nikishin \cite{N1,N2} and Maurey \cite{M73,M74}.\\
The Nikishin-Maurey-Pisier factorisation as in Theorem 1.7 and 2.4 in Chapter VI in \cite{pgc}, simplified for our purposes,  states that \eqref{pinf} implies the following so-called Nikishin condition:\\
For any $\e>0$, there exist a constant $C_\epsilon$ and a set $E\subset S$ such that
\begin{align*}
		\sigma(S\backslash E)<&\epsilon \qquad\text{and}\\
		\sigma(\xi\in E:|Rf(\xi)|>\lambda)\leq& C_\epsilon \left(\frac{\|f\|_p}{\lambda}\right)^p
\end{align*}
for all $f\in L^p(\R^{n+1})$ and $\lambda>0$. This is almost a weak type $p$ estimate, but for the exceptional set $S\backslash E$, which is small. So far the argument works 
even for the aforementioned more general operators. \\
The proof of the Nikishin-Maurey-Pisier factorisation in \cite{pgc} uses as two major tools: Firstly Khintchine's inequality, and secondly Zorn's Lemma. Both will in some sense feature in our proof as well, but only in a very simplified form. The complexity of our "Khintchine inequality" will be at the level of the parallelogram identity, and instead of the Lemma of Zorn or the equivalent axiom of choice, we will simply have to choose a representative from a set of at most two elements.
The reason for this perceived simplification is that we apply them repeatedly, as we will use an inductive argument. We believe that this argument is arguably easier accessible and more natural, at least in the Fourier restriction community, where induction on scales methods are standard by now.\\

Apart from this Nikishin-Maurey-Pisier factorisation, to deal with the above exceptional set, an extra invariance argument is needed. For the surface in question being the sphere, this crucial final step then relies on the rotational invariance of spheres, see Bourgain \cite{Bo1}. According to Bourgain in the same paper, this generalises to the paraboloid, although he does not give a proof and it seems less obvious to the author. The unbounded paraboloid comes with invariance under affine transformations, but restriction estimate \eqref{pinf} is known to fail for the unbounded paraboloid due to scaling considerations. A compact part of the paraboloid, on the other hand, is not invariant under affine transformations.\\
However, our approach is somewhat different, and we will be able to adapt the invariance argument to a quadratic surface setting to meet our purposes in Section \ref{sec:invar}.\\

\subsection{Localisation}
It is a by now standard to consider local estimates, that is, instead of attacking restriction estimates
\begin{align}\label{globrest}
	\|\E f\|_{L^p(\R^n)}\leq C \|f\|_{L^q(S)},
\end{align}
directly, one deals with the localised versions
\begin{align}\label{locrest}
	\|\E f\|_{L^p(B_R)}\leq C_\epsilon R^\epsilon \|f\|_{L^q(S)}, 
\end{align}
where $B_R$ is a ball  of radius $R$. Note that the estimate is invariant in the centre of $B_R$, as a translation of the centre only amounts to a harmless modulation of $f$. Indeed, by certain $\e$-removal techniques, \eqref{globrest} essentially (up to an endpoint) reduces to proving \eqref{locrest} for arbitrary small $\e$ , see for instance Theorem 1.2 in \cite{T99} or Theorem 5.3 in \cite{k17} for a later version.\\
At the first glance stronger, but indeed equivalent to the local restriction estimate is
\begin{align}\label{locrest2}
	\|\E f\|_{L^p(T_R)}\leq C_\epsilon R^\epsilon \|f\|_{L^q(S)}, 
\end{align}
where $T_R$ now is the horizontal plate $T_R=\R^n\times[-R,R]$ of thickness $2R$. 
Clearly \eqref{locrest2} implies \eqref{locrest}, while the converse uses certain orthogonality arguments. A detailed proof can be found in Lemma 5.1 in \cite{bmv20}.\\

\subsection{The main result}

In order to present our main result, let us introduce some notation.
Let  $\Omega=[0,1]^n$ and $S=\{(\xi,Q(\xi))|\xi\in\Omega\}$, where $Q$ is any quadratic form on $\R^n$. There is no restriction on the signature of $Q$, nor do we require any bounds on the eigenvalues of $Q$.\\
 We will call the surface $S$ a quadric and frequently identify a function $f\in L^1(S)$ with the function $\xi\to f(\xi,Q(\xi))$ defined on $\Omega$, and write $f(\xi)$ when we mean $f(\xi,Q(\xi))$. With the same understanding, we will also write $L^q(S)$ where we mean $L^q(\Omega)$.
Let $\E$ be the adjoint restriction or extension operator associated to $S$, that is
$$\E(f)(x,t)=\int_\Omega f(\xi)e^{ix\cdot\xi+tQ(\xi)}\dd\sigma(\xi).$$
For convenience, we will denote the adjoint restriction operator
 for the whole quadric\\ $\{(\xi,Q(\xi))|\xi\in\R^n\}$ by $\E$ as well whenever that does not lead to any ambiguity.\\

Since we are aiming at a restricted weak type estimate, we will restrict ourselves to a sufficiently rich subset of $L^q(S)$-functions, including all characteristic functions. However, in order to allow more flexibility, we will also allow signs. 
We define the set of signed characteristic functions: $$D:=\{f:\Omega\to\{0\}\cup\S^1|f\text{ is measurable}\},$$
where $\S^1:=\{z\in\CC:|z|=1\}$.

 For $f\in D$, define the "carrier" $\supp f:=\{f\neq0\}=\{|f|=1\}$ of $f\in D$. We will refer to $\mass(f):=|\supp f|$ as the "mass" of $f$

Note that $D$ is not a vector space, as for $f,g\in D$, in general $f+g\not\in D$, since their carriers might overlap. However, still $(f+g)\chi_{(\supp f\cap\supp g)^c}\in D$, and that will suffice for our considerations. Some further simple remarks about the set $D$:

\begin{remark}
Let $f,g\in D$, $E\subset\Omega$ measurable, and $\psi$ a measurable and real-valued function on $\Omega$. Then
\begin{enumerate}
\item $fg\in D$
\item  $-f\in D$
\item $e^{i\psi(\cdot)}f\in D$ 
\item$f\chi_E\in D$.
\end{enumerate}
\end{remark}

We fix $p$. For $1\leq q\leq\infty$ and $T\subset\R^{n+1}$ measurable, let $\C_q(T)$ denote the best constant $C$ such that
$$\|\E f\|_{L^p(T)}\leq C \|f\|_{L^q(S)} $$
for all $f\in D$. 
Here, we will consider horizontal plates $$T_R=\R^n\times[-R,R]$$ of thickness $2R$. Abusing notation, we will write $\C_q(R)$ when we mean $\C_q(T_R)$.\\
By trivial considerations, $\C_q(R)\lesssim R^{1/p}<\infty$ for all $q\geq p'$.

Note that if $I$ is any interval of length $|I|=2R$, then
$$\C_q(\R^n\times I) = \C_q(\R^n\times [-R,R]) $$
 since $\E f(x,t+s)=\E(e^{isQ}f)$ and $D$ is closed under modulations (iii).

Another small but important observation is: Since $[-2R,2R]$ is the union of the two intervals $[-2R,0]$ and $[0,2R]$ of length $2R$, we have the following doubling property:
\begin{remark}\label{double}
\begin{align}
	\C_q(2R)\leq 2^{\frac1p}\C_q(R).
\end{align}
\end{remark}

We now present our main result.

\begin{thm}\label{mainthm}  For any $p< q\leq\infty$ there exists a constant $C_1>1$ such that for all $R>1$
\begin{align}
	\C_\infty(R) \leq \C_q(R) \leq C_1 \C_\infty(R).
\end{align}
\end{thm}
For convenience, we formulate a version that is more along the lines of the usual language in restriction theory:
\begin{cor}
Let $p_0>(2n+2)/n$. Assume we have the restriction estimate
$$\|\E g\|_{L^{p_0}(\R^{n+1})}\leq C \|g\|_\infty$$
for all $g\in D\subset L^\infty(S)$. Then for all $p>p_0$, and $g\in L^p(S)$,
$$\|\E g\|_{L^p(\R^{n+1})}\leq C' \|g\|_p.$$
\end{cor}    

\begin{proof}
The assumptions of the corollary imply in particular
$$\C_\infty(T_R) \leq C$$
for all $R>1$.
But then $$\C_p(T_R) \leq CC_1^;$$
for all $p>p_0$, and since $D$ contains all characteristic functions $\chi_E$ of measurable sets $E$, we obtain for $R\to\infty$
$$\|\E \chi_E\|_{L^{p_0}(\R^{n+1})}\leq CC_1|E|^{1/p}.$$
Marcinkiewicz interpolation with the trivial $L^1(S)\to L^\infty(\R^{n+1})$ result gives the claim.
\end{proof}

\begin{remark}
Though we discuss the proof in the setting of our quadrics, an analogous argument applies for the sphere, giving an alternative proof for the factorisation argument for Fourier restriction to the sphere. Actually, for the sphere, our argument can be simplified compared to the case for quadrics. Since Lemma \ref{rotorgin} does not require the extra parameter $\delta$ from Lemma \ref{rot}, there is no need for the rescaling argument.
\end{remark}

\section{Further remarks and Preliminaries}

\subsection{Limitations of factorisation arguments in restriction theory}
The proof we present works for any quadratic phase function, regardless of its signature. Perturbations are a different matter. Although some kind of factorisation theorem might be conjectured to hold for perturbations of quadratic forms, i.e., compact surfaces with non-vanishing curvature, a proof would require further ideas, as they are not affine invariant.\\
Going beyond surfaces with non-vanishing curvature however, factorisation may fail. Müller, Vargas and the author studied restriction estimates for surfaces of finite type, the prototypical example being the graph $S_{m,l}$ of the phase function $\psi_{m,l}(x,y)=x^m+y^l$, $x,y\in[0,1]^2$. The range of admissible exponents can be described in terms of the so-called \emph{height} $h$, in this case given by $\frac{1}{h}=\frac{1}{m}+\frac{1}{l}$ and $M=\max\{m,l\}$.
The following holds true \cite{bmv}:

\begin{thm}[Buschenhenke, Müller, Vargas]
	Let $p>\frac{10}{3}$. Then the extension operator associated to 
	$\psi_{m,l}$ is bounded from $L^q(S_{m,l})$ to $L^p(\R^{n+1})$ if
	and only if $p>h+1$, $\frac{1}{q'}\geq\frac{h+1}{p}$ and 
	$\frac{1}{q}+\frac{2M+1}{p}<\frac{M+2}{2}$.
\end{thm}
The condition $p>\frac{10}{3}$ is not sharp, but the natural limit for the bilinear method used in \cite{bmv}. We emphasize that the condition $\frac{1}{q}+\frac{2M+1}{p}<\frac{M+2}{2}$ does not allow a factorisation theorem.
We give a concrete example where indeed \eqref{pinf} holds for certain values of $p$, but \eqref{pp} does not:

\begin{example}
Let $l=2$ and $m=M=7$. 
We see that we have boundedness from $L^\infty(S_{m,l})$ to $L^p(\R^{n+1})$ for all $p>\frac{10}{3}$; however, we have no boundedness from $L^p(S_{m,l})$ to $L^p(\R^{n+1})$ unless even $p>\frac{32}{9}>\frac{10}{3}:$\\
Here, the height is $h=\frac{14}{9}$, so we have $p>\frac{10}{3}>h+1$ in any case.
The critical condition is $\frac{1}{q}+\frac{15}{p}<\frac{9}{2}$, which for $q=\infty$ gives $p>\frac{10}{3}$, but for $q=p$, necessarily $p>\frac{32}{9}$.\\
\end{example}

\subsection{Invariance considerations}\label{sec:invar}
The factorisation argument in \cite{BoG} relies on the rotation invariance of the sphere. This rotation invariance can be exploited in the following lemma:
\begin{lemma}\label{rotorgin}
Let $E,F\subset\S^{n-1}$ measurable. Then there exist a rotation $\rho\in\mathcal{O}(n)$ such that 
$$ \sigma(E\cap\rho(F))= \sigma(E)\,\sigma(F).$$
\end{lemma}
Since we will not directly apply this lemma, but use a variation adapted to our quadric, namely Lemma \ref{rot}, we will leave the (very similar) proof to the reader. \\

Our quadric does not come with such rotation invariance; instead, we will make use of translations in parameter space, amounting to affine transformations of the surface. The obvious problem is that while $\R^n$ is translation-invariant, the compact subset $\Omega\subset \R^n$ is not. The following Lemma will meet our purposes:

\begin{lemma}\label{rot}
Let $\delta>0$ and $E,F\subset\R^n$ measurable. Then there exist a $\lambda\in[0,\delta]^n$ such that 
$$ |E\cap(F+\lambda)|\leq \delta^{-n}|E|\,|F|.$$
\end{lemma}

\begin{remark} We will apply the lemma for $E=F\subset \Omega$ with $|E|=|F|\ll 1$. Observe that in general $F+\lambda$ is not a subset of $\Omega$ anymore, but we only know that $F+\lambda$ is contained in $(1+\delta)\Omega$, as is $E$. We will deal with these issue by rescaling. Furthermore, we have to account for the parameter $\delta$ and accommodate a blow-up by $\delta^{-n}\gg1$ on the right-hand-side. This has to be balanced against the smallness of $|E|=|F|\ll 1$. We will give more details later in the proofs.
\end{remark}

\begin{proof} We have
\begin{align}
	|E|\,|F|=&\int_{\R^n}\int_{\R^n}\chi_E(x)\chi_F(x-y)\dd x\dd y\\
	=&\int_{\R^n}|E\cap(F+y)|\dd y \geq \int_{[0,\delta]^n} |E\cap(F+y)|\dd y.
\end{align}
By pigeonholing, there exist a $\lambda\in[0,\delta]^n$ such that 
$$|E|\,|F| \geq \delta^n |E\cap(F+\lambda)|,$$
as we claimed.
\end{proof}

\subsection{Clarkson's inequality}
Another tool we are going to use is the so-called Clarkson's inequality \cite{Cl}. 
\begin{lemma}\label{clark}
Let $(X,\mu)$ be a measure space, and $2\leq p<\infty$. Then for all $F,G\in L^p(X)$, we have for $p<\infty$
\begin{align}
	\|F\|_{L^p(X)}^p + \|G\|_{L^p(X)}^p  \leq \frac{\|F+G\|_{L^p(X)}^p + \|F-G\|_{L^p(X)}^p}{2}. 
\end{align}
\end{lemma}
\begin{remark} There is a analogue version for the case $p=\infty$, and Clarkson has a similar, though slightly more complicated formula for $p<2$, but since we are only interested in the case $2\leq p<\infty$, we refer the reader to \cite{Cl} or \cite{lt} for further information.
\end{remark}
For sake of completeness, we sketch the short proof.
\begin{proof}
We begin by taking the root of the well-known parallelogram identity: For $z,w\in\CC$ we have
\begin{align}
	(|z|^2+|w|^2)^\frac12
	= \left(\frac{|z+w|^2+|z-w|^2}{2}\right)^\frac12.
\end{align}
Since $p\geq 2$, we can use the embeddings of $\ell^p$-spaces, for the left-hand side with the counting measure, for the right-hand side with probability measure to conclude that
\begin{align*}
	(|z|^p+|w|^p)^\frac1p
	\leq \left(\frac{|z+w|^p+|z-w|^p}{2}\right)^\frac1p.
\end{align*}
Application to $z=F(x)$ and $w=G(x)$ gives
\begin{align*}
	\|F\|_{L^p(X)}^p + \|G\|_{L^p(X)}^p
	=& \int_X |F(x)|^p+|G(x)|^p d\mu(x)	\\
	\leq& \frac12 \int_X |F(x)+G(x)|^p+|F(x)-G(x)|^p d\mu(x)	\\
	=& \frac12 \|F+G\|_{L^p(X)}^p +\frac12 \|F-G\|_{L^p(X)}^p,
\end{align*}
which is what we claimed.
\end{proof}

\medskip


\section{Induction on scales}

\subsection{Near extremisers}
Given an error $\e>0$, we will call a function $f\in D$ an $\e$-near extremiser for $\C_q(R)$, if
\begin{align}
	\|\E f\|_{L^p(T_R)}\geq(1-\e) \C_q(R) \|f\|_{L^q(S)}. 
\end{align}
Clearly, $\e$-near extremisers always exist.

Note that if $f\in D$ is such an $\e$-near extremiser for $\C_q$, then for any $1\leq \tilde q\leq\infty$, we have
$$(1-\e) \C_q(R) \|f\|_{L^q(S)}\leq \|\E f\|_{L^p(T_R)} \leq \C_{\tilde q}(R) \|f\|_{L^{\tilde q}(S)}.$$
If there would be an $\e$-near extremiser $f$ for $\C_q(R)$ with $|\supp f|=1$, then 
$\C_q(R)$ and $\C_{\tilde q}(R)$ would be comparable up to a factor $(1-\e)$.
Thus our aim will be to construct near extremisers with large mass $|\supp f|$. 

Furthermore, for $\e>0,\alpha\geq1$, we will call $f\in D$ a $(\e,\alpha)$-near extremiser for $\C_q(R)$, if
\begin{align}
	\|\E f\|_{L^p(T_{\alpha R})}\geq(1-\e) \C_q(R) \|f\|_{L^q(S)}. 
\end{align}
Clearly, any $\e$-near extremiser is a $(\e,\alpha)$-near extremiser.

%
%
%
%

Our key result is the following lemma, which essentially allows us to increase the mass of a near extremiser:

\begin{lemma}\label{keylem1}
Let $2\leq p,q\leq\infty$, $\delta>0$, $0<\kappa<1<\alpha<2 $. Assume that $f\in D$ is a $(1-\kappa,\alpha)$-near extremiser for $\C_q(R)$, that is
\begin{align}
	\|\E f\|_{L^p(T_{\alpha R})}\geq\kappa \C_q(R) \|f\|_{L^q(S)}.
\end{align}
Then there exists a $ f_1\in D$ such that
\begin{align}
	2\,\mass(f) \geq \mass(f_1) \geq 2\,\mass(f)
	\left(1-\frac{\mass(f)}{2\delta^n}\right)	\frac{1}{(1+\delta)^n}
\end{align}
and
\begin{align}
	\|\E f_1\|_{L^p(T_{(1+\delta)^2\alpha R})}
	\geq 2^{\frac{1}{p}-\frac1q}(1+\delta)^{\frac{n+2}{p}-\frac{n}{q'}}\left(\kappa-\left[\frac{\mass(f)}{\delta^n}\right]^\frac1q\right) \C_q(R) \| f_1\|_{L^q(S)}.
\end{align}
\end{lemma}

The notation clears up if we restrict ourselves to the regime $p<q$. Then for sufficiently small $\delta>0$, 
\begin{align}\label{fact}
	2^{\frac{1}{p}-\frac1q}(1+\delta)^{\frac{n+2}{p}-\frac{n}{q'}} 
	\geq 2^{\frac{1}{p}-\frac1q}(1+\delta)^{-n} \geq 1.
\end{align}

\begin{lemma}\label{keylem}
For $p< q\leq\infty$ exist a $\delta_0>$ so that for all $0<\delta<\delta_0$and $0<\kappa<1<\alpha<2$, the following holds true. Assume that $f\in D$ is a $(1-\kappa,\alpha)$-near extremiser for $\C_q(R)$, that is
\begin{align}
	\|\E f\|_{L^p(T_{\alpha R})}\geq\kappa \C_q(R) \|f\|_{L^q(S)}.
\end{align}
Then there exists an $f_1\in D$ such that 
\begin{align}
	2\,\mass(f) \geq \mass(f_1) \geq 2\,\mass(f)
	\left(1-\frac{\mass(f)}{2\delta^n}\right)	\frac{1}{(1+\delta)^n}
\end{align}
and
\begin{align}
	\|\E f_1\|_{L^p(T_{(1+\delta)^2\alpha R})}\geq\left(\kappa-\left[\frac{\mass(f)}{\delta^n}\right]^\frac1q\right) \C_q(R) \| f_1\|_{L^q(S)}.
\end{align}
\end{lemma}

A number of comments are in order:\\
First of all, we use the notion of $(\e,\alpha)$-near extremisers, where we only assume a lower bound for the integral on the bigger plate $T_{rR}$ of thickness $2rR$, instead of the more natural $T_R$. This is because when passing to $f_1$, we have to increase the thickness by $(1+\delta)^2$ due to a rescaling. To allow an inductive application of the lemma, we make use of the parameter $\alpha>1$.\\
Secondly, one should think of $\delta$ being small, and take into account that the lemma is only effective if $\mass(f)\ll\delta^n$, so that $\kappa$ is shrinked not to much in the inductive process, while the mass of $f_1$ almost doubles compared to the mass of $f$.

\begin{figure}
\includegraphics[scale=0.5]{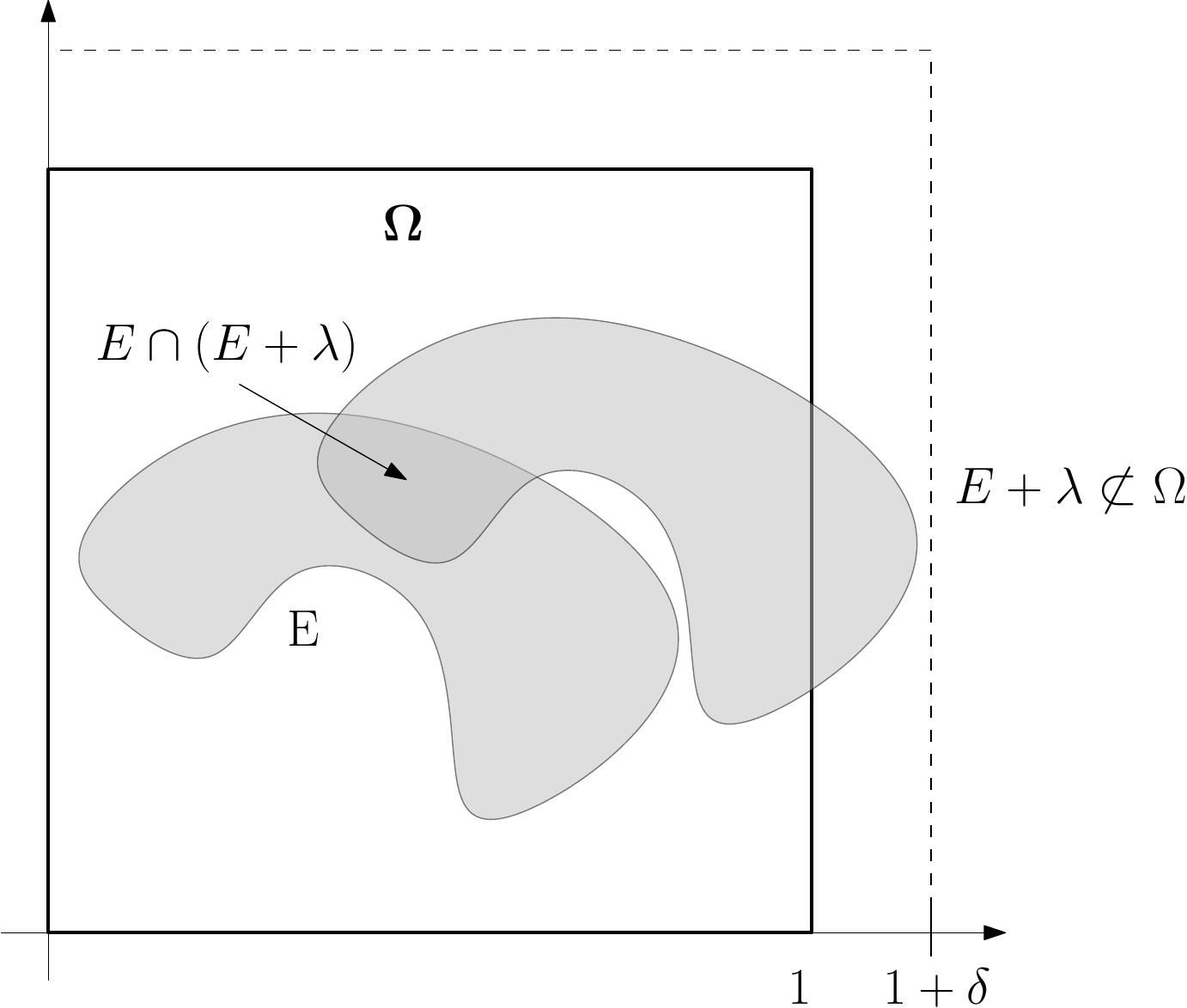}
\caption{Setup for the proof of Lemma \ref{keylem1}}
\end{figure}

Lemma \ref{keylem} follows from Lemma \ref{keylem1}, which we prove now.

\begin{proof}
Let denote $E:=\supp f$.
By Lemma \ref{rot}, we find some $\lambda\in[0,\delta]^n$ such that
for $E_\cap:=E\cap (\lambda+E)$ we have
\begin{align}
	|E_\cap|\leq \frac{|E|^2}{\delta^n}.
\end{align}
Now let $f_\lambda:=f(\cdot-\lambda)$ and
\begin{align}
	f_\pm:=f\pm f_\lambda(1-\chi_{E_\cap}).
\end{align}
Note that $f_\pm$ may not be supported in $\Omega$, but only in $(1+\delta)\Omega$ (cf. Fig. 2).
Observe that 
\begin{align}\label{unscsupp}
	|\supp f_\pm|= |E|+|(\lambda+E)\backslash E_\cap|
	\geq 2|E|-\frac{|E|^2}{\delta^n}
	\geq 2|E|\left(1-\frac{|E|}{2\delta^n}\right)
\end{align}
and
\begin{align}
	|\supp f_\pm| \leq  2|E|.
\end{align}
Furthermore, if $Q(\xi)=\xi^tA\xi$, then
\begin{align}
	|\E f_\lambda(x,t)|=|\int f(\xi-\lambda)e^{i(x\cdot\xi+tQ(\xi))}\dd\xi|
	=|\int f(\xi)e^{i([x+t(A+A^t)\lambda]\cdot\xi+tQ(\xi))}\dd\xi|
	=|\E f(x+t(A+A^t)\lambda,t)|,	
\end{align}
Since $T_{\alpha R}$ is invariant with respect to the map $(x,t)\to(x+t(A+A^t)\lambda,t)$, we see that 
$$\|\E f_\lambda\|_{L^p(T_{\alpha R})} = \|\E f\|_{L^p(T_{\alpha R})}.$$ 
Using the $(1-\kappa,\alpha)$-near extremiser property and Clarkson's inequality (Lemma \ref{clark}), we conclude that
\begin{align*}
	2^{\frac1p}\C_{q}(R)\kappa |E|^{\frac1q}
	\leq& 2^{\frac1p} \|\E f\|_{L^p(T_{\alpha R})}
	\\=& \left(\|\E f\|^p_{L^p(T_{\alpha R})}+\|\E f_\lambda\|^p_{L^p(T_{\alpha R})}\right)^\frac1p
	\\ \leq& \left(\frac12 \sum_\pm \|\E(f\pm f_\lambda)\|^p_{L^p(T_{\alpha R})}\right)^\frac1p . 
\end{align*}
Recalling that $f_\pm=f\pm f_\lambda(1-\chi_{E_\cap})$, we see that 
\begin{align*}
	2^{\frac1p}\C_{q}(R)\kappa |E|^{\frac1q}
	\leq& \left(\frac12 \sum_\pm \left[\|\E f_\pm\|_{L^p(T_{\alpha R})}+\|\E(f_\lambda\chi_{E_\cap})\|_{L^p(T_{\alpha R})}\right]^p\right)^\frac1p\\
	\leq
	& \left(\frac12 \sum_\pm \|\E f_\pm\|^p_{L^p(T_{\alpha R})}\right)^\frac1p
				+ \|\E(f_\lambda\chi_{E_\cap})\|_{L^p(T_{\alpha R})}.  
\end{align*}
Here, we used the triangle inquality for both the $L^p$-norm on $T_{\alpha R}$ and the outer $\ell_p$-norm.\\
We now bootstrap the error term $\|\E(f_\lambda\chi_{E_\cap})\|_{L^p(T_{\alpha R})}$. First observe that 
$$\supp(f_\lambda\chi_{E_\cap})\subset E_\cap\subset E\subset \Omega,$$
securing that $f_\lambda\chi_{E_\cap}\in D$. Therefore by Remark \ref{double}
\begin{align}
	\|\E(f_\lambda\chi_{E_\cap})\|_{L^p(T_{\alpha R})}
	\leq \|\E(f_\lambda\chi_{E_\cap})\|_{L^p(T_{2R})}
	\leq 2^{\frac1p}\C_q(R) |E_\cap|^{1/q}
	\leq 2^{\frac1p}\C_q(R)(\d^{-n}|E|^2)^{1/q},
\end{align}
and hence 
\begin{align}
	2^{\frac1p}\C_{q}(R) |E|^{\frac1q}(\kappa-\delta^{-n/q}|E|^{1/q})
	\leq \left(\frac12 \sum_\pm \|\E f_\pm \|^p_{L^p(T_{\alpha R})}\right)^\frac1p.
\end{align}
Since the left-hand side is a lower bound for the average of $\|\E f_\pm \|_{L^p(T_{\alpha R})}$ over the two signs $\pm$, it must be a lower bound for at least one of these terms.\\
Therefore, for either $\tilde f=f_+$ or $\tilde f=f_-$, we have
\begin{align}
	2^{\frac1p}\C_{q}(R) |E|^{\frac1q}(\kappa-\delta^{-n/q}|E|^{1/q})
	\leq \|\E \tilde f \|_{L^p(T_{\alpha R})},
\end{align}
and since $\|\tilde f\|_q \leq (2|E|)^{\frac1q}$, the claim follows by a scaling argument. Note that for $$f_1(\xi):=\tilde f((1+\delta)\xi),$$
while $\supp\tilde f\subset (1+\delta)\Omega$, we have 
$\supp f_1\subset\Omega$.

\end{proof}

We now proceed by iterating Lemma \ref{keylem}, increasing the mass of the near extremiser in every step:

\begin{lemma}\label{iter}
Let $p<q\leq\infty$.
For any $\e>0$, there exists a $\gamma>0$ such that for all $R>0$ the following holds:
If $f_0$ is a $\e/2$-extremiser for $\C_q(T)$ and $\mass f_0<\gamma$, then
there exists an $f\in D$ such that $\mass f\geq\gamma $ and $f$ is an $(\e,1+\e)$-near extremiser, that is,
\begin{align}
	\|\E(f)\|_{L^p(T_{(1+\e)R})}\geq (1-\e)\C_q(R) \|f\|_{L^q(S)}.
\end{align}
\end{lemma}

\begin{proof}
In the proof, we will choose a parameter $\beta>0$, $\beta\ll 1$, fulfilling several smallness conditions, that may depend on $\e$, but not on $R$, and then choose $\gamma:=\beta^{4n+2}$.\\
We assume that $f_0$ is an $\e/2$-near extremiser for $\C_q(R)$ and $|\supp f_0|\leq \beta^{4n+2}$.\\
Inductively, we will construct sequences of functions $f_j\in D$, $j=0,\ldots,J$ such that for $E_j=\supp f_j$ and the parameters $\delta_j$, defined by the formula
\begin{align}\label{delta}
	|E_j|^{1/2}=\beta\delta_j^n
\end{align}
we have 
\begin{align}\label{ONE}
\|\E f_J\|_{L^p(\prod_{j=0}^{J-1}(1+\delta_j)^2\cdot T)}
\geq \left(1-\frac{\e}{2}-\sum_{j=0}^{J-1}\left[\frac{|E_j|}{\delta_j^n}\right]^{1/q}\right) \C_q(T)\|f_j\|_q,
\end{align}
\begin{align}\label{TWO}
	|E_{j+1}|\geq \frac32 |E_{j}|,\qquad\forall 0\leq j<J
\end{align}
and
\begin{align}\label{THREE}
	|E_j|< \beta^{4n+2}\qquad\forall 0\leq j<J.
\end{align}
The aim is to construct such a sequence with
\begin{align}
	|E_{J}|\geq \beta^{4n+2}.
\end{align}
For $J=0$, \eqref{ONE} is clear, and \eqref{TWO} and \eqref{THREE} are void.
So, assume we constructed a sequence $f_0,\ldots,f_{J}$ so that \eqref{ONE}, \eqref{TWO} and \eqref{THREE} hold true.
If $|E_{J}|\geq \beta^{4n+2}$, we are done. If $|E_{J}|<\beta^{4n+2}$, we iterate. Observe first that since
$	|E_j|^{1/2}\beta^{-1} \leq \beta^{2n},$
we have 
\begin{align}\label{sonny}
	\delta_j\leq\beta^2< \delta_0
\end{align}
if $\beta$ is sufficiently small. Here, $\delta_0$ is the constant from Lemma \ref{keylem}.
We seek to apply Lemma \ref{keylem} to $f=f_J$ and the parameter $\delta=\delta_J$, with $\kappa=\kappa_J=1-\frac{\e}{2}-\sum\limits_{i=0}^{J-1}\left[\frac{|E_i|}{\delta_i^n}\right]^{1/q}$ and $\alpha=\alpha_J=\prod\limits_{i=0}^{J-1}(1+\delta_i)^2$. Therefore, we have to ensure $\kappa_J>0$ and $\alpha_J\leq 2$.
 
In fact, we will show at the end of the proof that
\begin{align}\label{alfa}
	\alpha_J\leq& 1+\e,\\
	\kappa_J\geq& 1-\e.\label{bravo}
\end{align}
provided \eqref{TWO} and \eqref{THREE} hold true. Taking that for granted for now, 
Lemma \ref{keylem} provides us with an $f_{J+1}\in D$ such that
\begin{align}\label{tica}
	|E_{J+1}|\geq 2|E_J|\left(1-\frac{|E_J|}{2\delta_J^n}\right)\frac1{(1+\delta_J)^n}
\end{align}
 and
\begin{align}\label{grebe}
\|\E f_J\|_{L^p((1+\delta_J)^2\alpha_J T_R)}
\geq \left(\kappa-\left[\frac{|E_J|}{\delta_J^n}\right]^{1/q}\right) \C_q(T)\|f_J\|_q.
\end{align}
We check that the sequence $f_0,\ldots,f_{J+1}$ fulfills \eqref{ONE}, \eqref{TWO} and \eqref{THREE}:
\eqref{ONE} is immediate, \eqref{THREE} by assumption. For \eqref{TWO},

since $\delta_j\leq\beta^2$, \eqref{tica} implies
\begin{align}
|E_{j+1}|\geq& |E_j|(2-\beta |E_j|^{1/2})(1+\delta_j)^{-n}\\
\geq& |E_j|(2-\beta)(1+\beta^2)^{-n}
\geq \frac32 |E_j|
\end{align}
provided $(2-\beta)(1+\beta^2)^{-n}\geq 3/2$. \\

We claim that the procedure stops after finitely many steps. If not, we would obtain an infinite sequence of $E_j$, which according to \eqref{THREE} would be bounded, contradicting \eqref{TWO}.\\

Hence there exist a finite $J$ such that \eqref{ONE} holds and $|E_{J}|\geq \beta^{4n+2}$. Using \eqref{alfa} and \eqref{bravo}, we see that \eqref{ONE} implies that $f_{J+1}$ is an $(\e,1+\e)$-near extremiser.\\

It remains to show that
\begin{align*}
	\alpha_J\leq& 1+\e,\\
	\kappa_J\geq& 1-\e
\end{align*}
provided \eqref{TWO} and \eqref{THREE} hold true.
Since by \eqref{TWO}, $|E_j|\leq \left(\frac23\right)^{J-1-j}|E_{J-1}|$, for any exponent $r>0$, we have
\begin{align}
	\sum_{j=0}^{J-1}|E_j|^r 
\leq& \sum_{j=0}^{J-1}\left(\frac23\right)^{r(J-1-j)}|E_{J-1}|^r
\leq c_r\beta^{(4n+2)r}.
\end{align}
For $r=1/(2q)$, we obtain
\begin{align}
	\sum_{j=0}^{J-1}\left[\frac{|E_j|}{\delta_j^n}\right]^{1/q}
	=\beta^{1/q} \sum_{j=0}^{J-1}|E_j|^{1/(2q)}
	\leq c_\frac{1}{2q}\beta^{(2n+2)/q} \leq \e/2,
\end{align}
provided $\beta$ is sufficiently small, and hence \eqref{bravo}.
Furthermore,
\begin{align*}
	\log(\prod_{j=0}^{J-1}(1+\delta_j)^2) 
	\leq& 2\sum_{j=0}^{J-1} \delta_j \\
	=& 2\beta^{-1/n}\sum_{j=0}^{J-1} |E_j|^{1/(2n)}\\
	\leq& 2c_\frac{1}{2n}\beta^{\frac{4n+2}{2n}-\frac1n}\\
	=& 2c_\frac{1}{2n}\beta^{2}\leq\beta,
\end{align*}
again for sufficiently small $\beta$.
Then
\begin{align}
	\alpha_J=\prod_{j=0}^{J-1}(1+\delta_j)^2
	\leq \exp(\beta)\leq 1+\e,
\end{align}
if $\beta$ is small enough with respect to $\e$.
\end{proof}

As an immediate consequence, we can deduce the existence of near-extremisers with large mass:

\begin{cor}\label{ext}
Let $\frac{2n+2}{n}\leq p\leq q<\infty$.
For any $\e>0$, there exists a $\gamma>0$ such that the following holds:
For all $R\geq1$ there exists an $f\in D$ such that $|\supp f|\geq\gamma $ and
\begin{align}
	\|\E f\|_{L^p((1+\e)T_R)}\geq (1-\e)\C_q(R) \|f\|_{L^q(S)}.
\end{align}
\end{cor}
\begin{proof}
Given $\e>0$, we choose $\gamma>0$ as in Lemma \eqref{iter}. Let $f_0$ be a $1-\e/2$-near extremiser of $\C_q(R)$. If $|\supp f_0|\geq\gamma$, we are done by choosing $f=f_0$.
If $|\supp f_0|<\gamma$, the aforementioned lemma applies.\\

\end{proof}

We can now prove our main result, Theorem $\ref{mainthm}$. We have to show that 
\begin{align}
	\C_\infty(R) \leq \C_q(R) \leq C_1^{1/q} \C_\infty(R).
\end{align}
\begin{proof}
The first inequality is a simple consequence of Hölders inequality. The second inequality is trivial for the case $q=\infty$, and for $q<\infty$, by an interpolation argument, it is enough to consider the case $q=p$.\\
Choose $\gamma>0$ according to Corollary \eqref{ext}, for $\e=1/2$ and $q=p$.
Then for all $R>0$, there exist an $f\in D$ such that
\begin{align}
	\gamma^{1/p}\C_p(R) \leq \|f\|_p\ \C_p(R)
	\leq 2\|\E f\|_{L^p(T_{2R})}.	
\end{align}
On the other hand,
\begin{align}
	\gamma^{1/p}\C_p(R) 
	\leq 2\|\E f\|_{L^p(T_{2R})} \leq 2 \C_\infty(2R)\leq 2^{1+1/p}\C_\infty(T),	
\end{align}
hence we can choose $C_1=\dfrac{2^{p+1}}{\gamma}$.
\end{proof}

\bigskip

\thispagestyle{empty}

\renewcommand{\refname}{References}

\end{document}